\documentclass[12pt]{article}
\usepackage{e-jc}

\usepackage{amsthm,amsmath,amssymb}

\usepackage{graphicx}

\usepackage[colorlinks=true,citecolor=black,linkcolor=black,urlcolor=blue]{hyperref}

\theoremstyle{plain}
\newtheorem{theorem}{Theorem}[section]

\newtheorem{corollary}[theorem]{Corollary}

\newtheorem{claim}[theorem]{Claim}

\theoremstyle{definition}

\theoremstyle{remark}

\numberwithin{equation}{section}
\numberwithin{figure}{section}
\numberwithin{table}{section}

\newcommand{\T}{\operatorname{T}}
\newcommand{\E}{\operatorname{E}}
\newcommand{\Od}{\operatorname{O}}

\newcommand{\U}{\operatorname{U}}
\newcommand{\D}{\operatorname{D}}

\title{A note on a 2-enumeration of antisymmetric monotone triangles}

\author{Tri Lai\footnote{This research was supported in part by the Institute for Mathematics and its Applications with funds provided by the National Science Foundation (grant no. DMS-0931945).}\\
\small Institute for Mathematics and its Applications\\[-0.8ex]
\small University of Minnesota\\[-0.8ex]
\small Minneapolis, MN 55455\\
\small Email: tmlai@ima.umn.edu\\
\small Website: \url{http://www.ima.umn.edu/~tmlai/}
}

\date{\small Mathematics Subject Classifications: 05A15, 05C70, 05E99}

\begin{document}
\maketitle


\begin{abstract}
In their unpublished work,   Jockusch and Propp showed that a 2-enumeration of antisymmetric monotone triangles is given by a simple product formula.  On the other hand, the author proved that the same formula counts the domino tilings of the quartered Aztec rectangle.  In this paper, we explain this phenomenon directly by building a correspondence between the antisymmetric monotone triangles and domino tilings of the quartered  Aztec rectangle. 

\end{abstract}
\section{Introduction}
A \textit{monotone triangle} of order $n$ is a triangular array of integers  
\begin{center}
\begin{tabular}{rccccccccc}
&                 &               &                 &                 &   $a_{1,1}$ \\\noalign{\smallskip\smallskip}
&                 &               &                 & $a_{2,1}$ &                   & $a_{2,2}$ \\\noalign{\smallskip\smallskip}
&                 &               & $a_{3,1}$ &                 &  $a_{3,2}$  &                & $a_{3,3}$\\\noalign{\smallskip\smallskip}
&                 &           \reflectbox{$\ddots$} &            &            \reflectbox{$\ddots$}         &        $\vdots$  &         $\ddots$       &           &    $\ddots$ \\\noalign{\smallskip\smallskip}
& $a_{n,1}$ &               &     $a_{n,2}$            &                  & $\dots$       &               &  $a_{n,n-1}$             &                & $a_{n,n}$\\\noalign{\smallskip\smallskip}
\end{tabular}
\end{center}
whose  entries are strictly increasing  along the rows and weakly increasing along both rising and descending diagonals from left to right.  An \textit{antisymmetric monotone triangle} (AMT) is a monotone triangle, which has $a_{i,k}=-a_{i,i+1-k}$, for any $1\leq i\leq n$ and $1\leq k\leq i$ (see  examples of AMTs in Figures \ref{AMT2}(b) and \ref{AMT6}(b)).

Let  $q$ be a number, then the \textit{$q$-weight} of an AMT is $q^{w}$ if it has $w$ positive entries, which do not appear on the row above. Assume that $n\geq 2$ and $0<a_1< a_2<\dotsc<a_{\lfloor\frac{n}{2}\rfloor}$, the \textit{$q$-enumeration} $A^{q}_{n}(a_1,a_2,\dotsc, a_{\lfloor\frac{n}{2}\rfloor})$ of AMTs is defined as the sum of q-weights of all AMTs of order $n$ whose positive entries on the bottommost row are $a_1,a_2,\dotsc,a_{\lfloor\frac{n}{2}\rfloor}$. Since there is only one AMT of order $1$ that consists of a $0$, we set $A_1^{q}(\emptyset):=1$, for any $q$.

Jockusch and Propp proved a simple product formula for the $2$-enumeration of the AMTs. 
\begin{theorem}[Jockusch and Propp \cite{JP}]\label{JPthm}
Assume that $k,a_1,a_2,\dotsc,a_k$ are positive integers, such that $a_1<a_2<\dotsc<a_k$. The AMTs with positive entries $a_1,a_2,\dotsc,a_k$ on the bottom are 2-enumerated by 
\begin{equation}\label{E}
A^{2}_{2k}(a_1,a_2,\dotsc,a_k)=\frac{2^{k^2}}{0!2!4!\dotsc(2k-2)!}\prod_{1\leq i<j\leq k}(a_j-a_i)\prod_{1\leq i<j\leq k}(a_i+a_j-1)
\end{equation}
and
\begin{equation}\label{O}
A^{2}_{2k+1}(a_1,a_2,\dotsc,a_k)=\frac{2^{k^2}}{1!3!5!\dotsc(2k-1)!}\prod_{1\leq i<j \leq k}(a_j-a_i)\prod_{1\leq i\leq j \leq k}(a_i+a_j-1),
\end{equation}
where the empty products (like $\prod_{1\leq i<j\leq k}(a_j-a_i)$ for $k=1$) equal 1 by convention.
\end{theorem}
We denote by  $\E(a_1,a_2,\dotsc,a_k)$ and  $\Od(a_1,a_2,\dotsc,a_k)$ the expressions on the right-hand sides of (\ref{E}) and (\ref{O}), respectively. The Jockusch-Propp's proof of 
Theorem \ref{JPthm} is algebraic, and no combinatorial proof has been known.

\medskip

Let $R$ be a (finite, connected) region on the square lattice. A \textit{domino tiling} of $R$ is a covering of $R$ by dominoes so that there are no gaps or overlaps. We use the notation $\T(R)$ for the number of domino tilings of the region $R$.

The \textit{Aztec diamond} of order $n$ is the union of all unit squares inside the contour $|x|+|y|= n+1$. The Aztec diamond  of order $9$ is shown in Figure \ref{QA}. It has been proven that the number of domino tilings of the Aztec diamond of order $n$ is $2^{\frac{n(n+1)}{2}}$. See \cite{Elkies} for the four original proofs; and e.g.  \cite{Bosio},  \cite{Brualdi},  \cite{Eu},  \cite{Fen}, \cite{Kamioka}, \cite{Kuo}, \cite{propp} for further proofs.

Jockusch and Propp \cite{JP} introduced three types of \textit{quartered Aztec diamonds} (denoted by $R(n)$, $K_a(n)$, and $K_{na}(n)$) obtained by dividing the Aztec diamond of order $n$ into four parts by two zigzag cuts passing the center as in Figure \ref{QA} (for $n=9$). They proved that the domino tilings of the quartered Aztec diamond are enumerated by $\E(a_1,a_2,\dotsc,a_k)$ or  $\Od(a_1,a_2,\dotsc,a_k)$, where $a_1,a_2,\dotsc,a_k$ are the first $k$ odd positive integers or the first $k$ even positive integers. The author gave a new proof for this result by using Ciucu's Factorization Theorem (see \cite[Theorem 1.2]{Ciucu}) in \cite{Trib}. 

\begin{figure}\centering
\setlength{\unitlength}{3947sp}%
\begingroup\makeatletter\ifx\SetFigFont\undefined%
\gdef\SetFigFont#1#2#3#4#5{%
  \reset@font\fontsize{#1}{#2pt}%
  \fontfamily{#3}\fontseries{#4}\fontshape{#5}%
  \selectfont}%
\fi\endgroup%
\resizebox{10cm}{!}{
\begin{picture}(0,0)%
\includegraphics{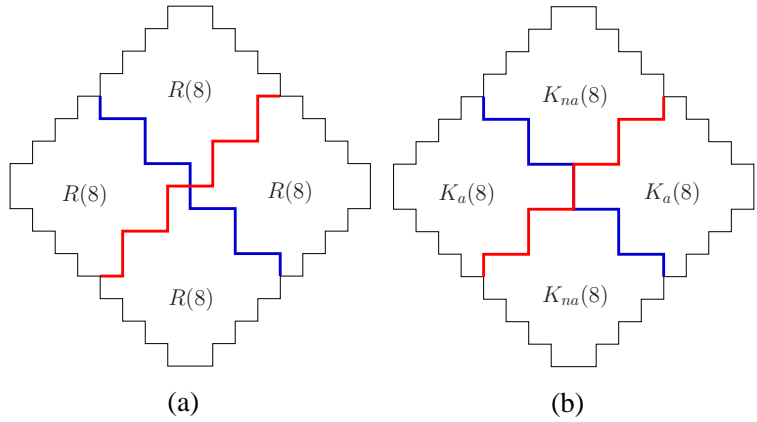}%
\end{picture}%
%
%

\begin{picture}(7827,4375)(698,-3997)
\put(6290,-636){\makebox(0,0)[lb]{\smash{{\SetFigFont{14}{16.8}{\rmdefault}{\mddefault}{\updefault}{$K_{na}(8)$}%
}}}}
\put(1254,-1681){\makebox(0,0)[lb]{\smash{{\SetFigFont{14}{16.8}{\rmdefault}{\mddefault}{\updefault}{$R(8)$}%
}}}}
\put(2357,-583){\makebox(0,0)[lb]{\smash{{\SetFigFont{14}{16.8}{\rmdefault}{\mddefault}{\updefault}{$R(8)$}%
}}}}
\put(3407,-1648){\makebox(0,0)[lb]{\smash{{\SetFigFont{14}{16.8}{\rmdefault}{\mddefault}{\updefault}{$R(8)$}%
}}}}
\put(2372,-2766){\makebox(0,0)[lb]{\smash{{\SetFigFont{14}{16.8}{\rmdefault}{\mddefault}{\updefault}{$R(8)$}%
}}}}
\put(5206,-1666){\makebox(0,0)[lb]{\smash{{\SetFigFont{14}{16.8}{\rmdefault}{\mddefault}{\updefault}{$K_{a}(8)$}%
}}}}
\put(6294,-2739){\makebox(0,0)[lb]{\smash{{\SetFigFont{14}{16.8}{\rmdefault}{\mddefault}{\updefault}{$K_{na}(8)$}%
}}}}
\put(7359,-1659){\makebox(0,0)[lb]{\smash{{\SetFigFont{14}{16.8}{\rmdefault}{\mddefault}{\updefault}{$K_{a}(8)$}%
}}}}
\end{picture}}
\caption{Three kinds of quartered Aztec diamonds of order 9.}
\label{QA}
\end{figure}


\begin{figure}\centering
\setlength{\unitlength}{3947sp}%
\begingroup\makeatletter\ifx\SetFigFont\undefined%
\gdef\SetFigFont#1#2#3#4#5{%
  \reset@font\fontsize{#1}{#2pt}%
  \fontfamily{#3}\fontseries{#4}\fontshape{#5}%
  \selectfont}%
\fi\endgroup%
\resizebox{14cm}{!}{

\begin{picture}(0,0)%
\includegraphics{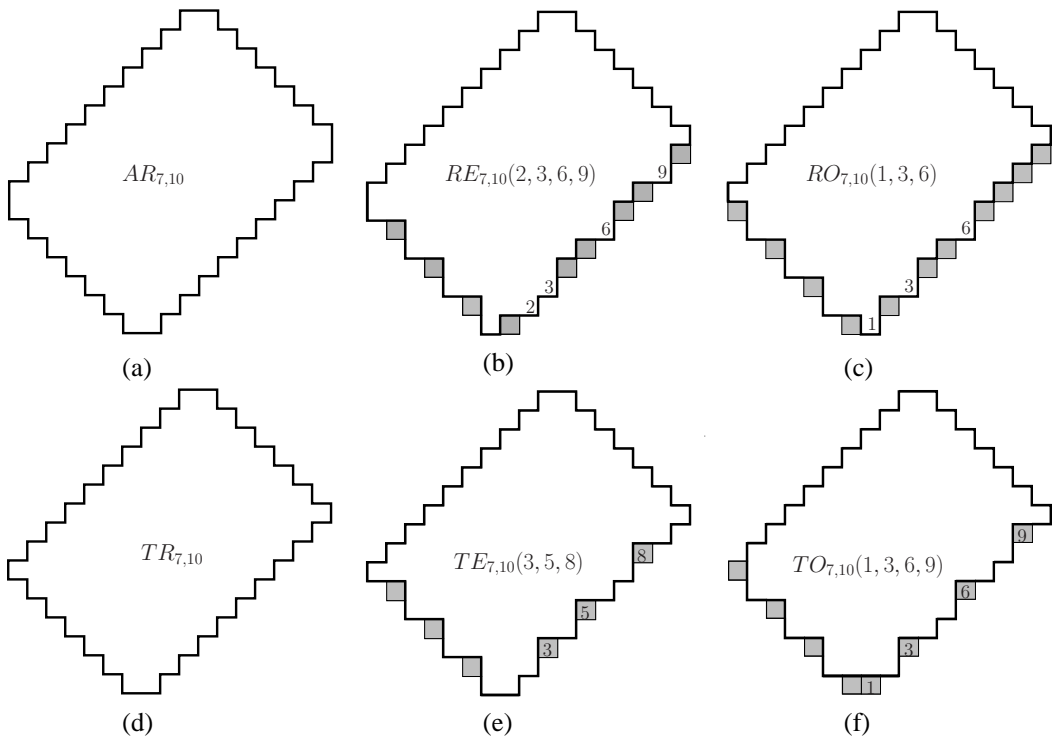}%
\end{picture}%
%
%

\begin{picture}(13035,9119)(440,-8960)
\put(10231,-6864){\makebox(0,0)[lb]{\smash{{\SetFigFont{17}{20.4}{\rmdefault}{\mddefault}{\updefault}{$TO_{7,10}(1,3,6,9)$}%
}}}}
\put(6911,-3635){\makebox(0,0)[lb]{\smash{{\SetFigFont{14}{16.8}{\rmdefault}{\mddefault}{\updefault}{2}%
}}}}
\put(7143,-3395){\makebox(0,0)[lb]{\smash{{\SetFigFont{14}{16.8}{\rmdefault}{\mddefault}{\updefault}{3}%
}}}}
\put(7852,-2675){\makebox(0,0)[lb]{\smash{{\SetFigFont{14}{16.8}{\rmdefault}{\mddefault}{\updefault}{6}%
}}}}
\put(8553,-1955){\makebox(0,0)[lb]{\smash{{\SetFigFont{14}{16.8}{\rmdefault}{\mddefault}{\updefault}{9}%
}}}}
\put(11164,-3852){\makebox(0,0)[lb]{\smash{{\SetFigFont{14}{16.8}{\rmdefault}{\mddefault}{\updefault}{1}%
}}}}
\put(11621,-3379){\makebox(0,0)[lb]{\smash{{\SetFigFont{14}{16.8}{\rmdefault}{\mddefault}{\updefault}{3}%
}}}}
\put(12334,-2659){\makebox(0,0)[lb]{\smash{{\SetFigFont{14}{16.8}{\rmdefault}{\mddefault}{\updefault}{6}%
}}}}
\put(7126,-7914){\makebox(0,0)[lb]{\smash{{\SetFigFont{14}{16.8}{\rmdefault}{\mddefault}{\updefault}{3}%
}}}}
\put(7591,-7449){\makebox(0,0)[lb]{\smash{{\SetFigFont{14}{16.8}{\rmdefault}{\mddefault}{\updefault}{5}%
}}}}
\put(8296,-6736){\makebox(0,0)[lb]{\smash{{\SetFigFont{14}{16.8}{\rmdefault}{\mddefault}{\updefault}{8}%
}}}}
\put(11146,-8379){\makebox(0,0)[lb]{\smash{{\SetFigFont{14}{16.8}{\rmdefault}{\mddefault}{\updefault}{1}%
}}}}
\put(11619,-7914){\makebox(0,0)[lb]{\smash{{\SetFigFont{14}{16.8}{\rmdefault}{\mddefault}{\updefault}{3}%
}}}}
\put(12324,-7194){\makebox(0,0)[lb]{\smash{{\SetFigFont{14}{16.8}{\rmdefault}{\mddefault}{\updefault}{6}%
}}}}
\put(13029,-6496){\makebox(0,0)[lb]{\smash{{\SetFigFont{14}{16.8}{\rmdefault}{\mddefault}{\updefault}{9}%
}}}}
\put(1891,-1996){\makebox(0,0)[lb]{\smash{{\SetFigFont{17}{20.4}{\rmdefault}{\mddefault}{\updefault}{$AR_{7,10}$}%
}}}}
\put(2127,-6720){\makebox(0,0)[lb]{\smash{{\SetFigFont{17}{20.4}{\rmdefault}{\mddefault}{\updefault}{$TR_{7,10}$}%
}}}}
\put(5907,-1995){\makebox(0,0)[lb]{\smash{{\SetFigFont{17}{20.4}{\rmdefault}{\mddefault}{\updefault}{$RE_{7,10}(2,3,6,9)$}%
}}}}
\put(10395,-1995){\makebox(0,0)[lb]{\smash{{\SetFigFont{17}{20.4}{\rmdefault}{\mddefault}{\updefault}{$RO_{7,10}(1,3,6)$}%
}}}}
\put(6016,-6826){\makebox(0,0)[lb]{\smash{{\SetFigFont{17}{20.4}{\rmdefault}{\mddefault}{\updefault}{$TE_{7,10}(3,5,8)$}%
}}}}
\end{picture}}
\caption{The Aztec rectangle, trimmed Aztec rectangle, and the four quartered Aztec rectangle of size $7\times 10$.}
\label{QAztec2}
\end{figure}

Besides the Aztec diamonds, we are also interested in a similar family of regions called \textit{Aztec rectangles} (see Figure \ref{QAztec2}(a) for the Aztec rectangle of size $7\times 10$). Denote by $AR_{m,n}$ the Aztec rectangle region having $m$ squares\footnote{From now on, we use the word \emph{square(s)} to mean \emph{unit square(s)}.} along the southwest side and $n$ squares along the northwest side. We also consider the \textit{trimmed Aztec rectangle} $TR_{m,n}$ obtained from the ordinary $AR_{m,n}$  by removing all squares running along its northwest and northeast sides (illustrated in Figure \ref{QAztec2}(d)).

We generalize the Jockusch-Propp's quartered Aztec diamonds as follows. Remove all squares at even positions (from the bottom to top) on the southwest side of $AR_{m,n}$, and remove  arbitrarily $n-\left\lfloor\frac{m+1}{2}\right\rfloor$  squares on the southeast sides. Assume that we are removing all squares, except for the $a_1$-st, the $a_2$-nd, $\dotsc$, and the $a_{\left\lfloor\frac{m+1}{2}\right\rfloor}$-th ones, from the southeast side, and denote by $RE_{m,n}(a_1,a_2,\dotsc,a_{\left\lfloor\frac{m+1}{2}\right\rfloor})$ the resulting region (see Figure \ref{QAztec2}(b)). We also have an odd-analog $RO_{m,n}(a_1,a_2,\dotsc,a_{\left\lfloor\frac{m}{2}\right\rfloor})$ of the above region when removing odd squares (instead of the even ones) from the southwest side, and removing all squares from the southeast side, except for the squares at the positions $a_1,a_2,\dotsc,a_{\left\lfloor\frac{m}{2}\right\rfloor}$  (illustrated in Figure \ref{QAztec2}(c)).

If we remove all even squares on the southwest side of the trimmed Aztec rectangle $TR_{m,n}$, and also remove the squares at the positions $a_1,a_2,\dotsc,a_{\left\lfloor\frac{m}{2}\right\rfloor}$ from its southeast side, we get the region $TE_{m,n}(a_1,a_2,\dotsc,a_{\left\lfloor\frac{m}{2}\right\rfloor})$ (shown in Figure \ref{QAztec2}(e)). Repeating the process with the odd squares on the southwest side removed, we get the region $TO_{m,n}(a_1,a_2,\dotsc,a_{\left\lfloor\frac{m+1}{2}\right\rfloor})$ (pictured in Figure \ref{QAztec2}(f)). 

We call each of the above four regions a \textit{quartered Aztec rectangle} (QAR). We notice that the quartered Aztec diamonds of order $2k$ are obtained from the $RE$- and $RO$-QARs by specializing $m=n=k$ and $a_i=2i$ or $2i-1$, depending on the region considered; and the quartered Aztec diamonds of order $2k+1$ are obtained similarly from the $TE$- and $TO$-QARs having size $(k+1)\times(k+1)$. 

Surprisingly, the functions  $\E(a_1,a_2,\dotsc,a_k)$ and $\Od(a_1,a_2,\dotsc,a_k)$ in Theorem \ref{JPthm} also counts the domino tilings of the QARs.
\begin{theorem}[Lai \cite{Tri}]\label{Trithm} For any $1\leq k< n$ and $1\leq a_1<a_2<\dotsc<a_k\leq n$ the domino tilings of QARs are enumerated by
\begin{equation}
\T(RE_{2k-1,n}(a_1,a_2,\dotsc,a_k))=\T(RE_{2k,n}(a_1,a_2,\dotsc,a_k))=\E(a_1,a_2,\dotsc,a_k),
\end{equation}
\begin{equation}
\T(RO_{2k,n}(a_1,a_2,\dotsc,a_k))=\T(RO_{2k+1,n}(a_1,a_2,\dotsc,a_k))=\Od(a_1,a_2,\dotsc,a_k),
\end{equation}
\begin{equation}
\T(TE_{2k,n}(a_1,a_2,\dotsc,a_{k}))=\T(TE_{2k+1,n}(a_1,a_2,\dotsc,a_{k}))=2^{-k}\Od(a_1,a_2,\dotsc,a_k),
\end{equation}
\begin{equation}
\T(TO_{2k-1,n}(a_1,a_2,\dotsc,a_{k}))=\T(TO_{2k,n}(a_1,a_2,\dotsc,a_{k}))=2^{-k}\E(a_1,a_2,\dotsc,a_k).
\end{equation}
\end{theorem}

Theorems \ref{JPthm} and \ref{Trithm} imply that
\begin{corollary}\label{coro} For any $1\leq k< n$ and $1\leq a_1<a_2<\dotsc<a_k\leq n$ 
\begin{equation}\label{main1}
\T(RE_{2k-1,n}(a_1,a_2,\dotsc,a_k))=\T(RE_{2k,n}(a_1,a_2,\dotsc,a_k))=A^{2}_{2k}(a_1,a_2,\dotsc,a_k),
\end{equation}
\begin{equation}\label{main2}
\T(RO_{2k,n}(a_1,a_2,\dotsc,a_k))=\T(RO_{2k+1,n}(a_1,a_2,\dotsc,a_k))=A^{2}_{2k+1}(a_1,a_2,\dotsc,a_k),
\end{equation}
\begin{equation}\label{main3}
\T(TE_{2k,n}(a_1,a_2,\dotsc,a_{k}))=\T(TE_{2k+1,n}(a_1,a_2,\dotsc,a_{k}))=2^{-k}A^{2}_{2k+1}(a_1,a_2,\dotsc,a_k),
\end{equation}
\begin{equation}\label{main4}
\T(TO_{2k-1,n}(a_1,a_2,\dotsc,a_{k}))=\T(TO_{2k,n}(a_1,a_2,\dotsc,a_{k}))=2^{-k}A^{2}_{2k}(a_1,a_2,\dotsc,a_k).
\end{equation}
\end{corollary}
In this paper,  we give a direct combinatorial proof of the identities (\ref{main1})--(\ref{main4}) by giving a correspondence between the AMTs and the domino tilings of the QARs.  The proof extends an idea of Jockusch and Propp for the case of quartered Aztec diamonds. On the other hand, our combinatorial proof and Theorem \ref{Trithm} yield a combinatorial proof for Jockusch-Propp's Theorem \ref{JPthm}.

Finally, it is worth to notice that Ciucu \cite[Theore 4.1]{Ciucu} showed a related correspondence between the domino tilings of the \emph{Aztec rectangle with holes} (i.e. Aztec rectangle with certain squares removed) and a 2-enumeration of the so-called \emph{alternating sign matrices}, which in turn are in bijection with the monotone triangles (see \cite{Mill}). However, the 2-enumeration of alternating sign matrices (or monotone triangles) does \emph{not} imply the 2-enumeration of AMTs in Theorem \ref{JPthm}; and the enumeration of tilings of an Aztec rectangle with holes does \emph{not} imply directly the  tiling formula of the QARs either.

\section{Direct proof of the identities (\ref{main1})--(\ref{main4})}

We  prove only the equalities (\ref{main1}) and (\ref{main3}), as (\ref{main2}) and (\ref{main4}) can be obtained by a perfectly analogous manner.

In this proof, we always rotate the QARs by $45^{\circ}$ clockwise to help the visualization of our arguments. First, we can see that the dominoes on the top of  $RE_{2k,n}(a_1,a_2,\dotsc,a_k)$ are all forced; and by removing these dominoes, we get the region $RE_{2k-1,n}(a_1,a_2,\dotsc,a_k)$. This implies that the two QARs in (\ref{main1}) have the same number of domino tilings. Thus, we only need to show that
\begin{equation*}
\T(RE_{2k-1,n}(a_1,a_2,\dotsc,a_k))=A^{2}_{2k}(a_1,a_2,\dotsc,a_k).
\end{equation*}

Denote by $\mathcal{T}(R)$ the set of all tilings of a region $R$, and $\mathcal{A}_{n}(a_1,a_2,\dotsc,a_{\lfloor\frac{n}{2}\rfloor})$ the set of all AMTs of order $n$ having positive entries $a_1<a_2<\dotsc<a_{\lfloor\frac{n}{2}\rfloor}$ on the bottom. 

Next, we define a map
\[\Phi: \mathcal{T}(RE_{2k-1,n}(a_1,a_2,\dotsc,a_k))\rightarrow\mathcal{A}_{2k}(a_1,a_2,\dotsc,a_k)\]
as follows. 
\begin{figure}\centering
\includegraphics[width=14cm]{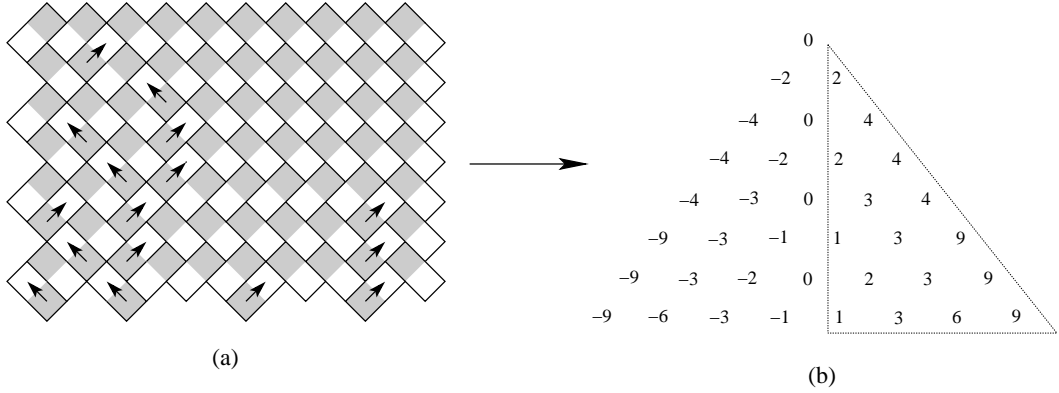}%
\caption{A domino tiling of $RE_{7,10}(2,3,6,9)$ and the corresponding AMT.}
\label{AMT2}
\end{figure}

Color the region $RE_{2k-1,n}(a_1,a_2,\dotsc,a_k)$ black and white so that any two squares sharing an edge have different colors, and that the bottommost squares are black.
Given a tiling $T$ of  the region. Figure \ref{AMT2}(a) shows a sample domino tiling of the QAR (with $k=4$, $n=10$, $a_1=1$, $a_2=3$, $a_3=6$, and $a_4=9$). We say a black square is \textit{matched upward} or \textit{matched downward}, depending on whether the white square covered by the same domino is above or below it. We use the arrows to show in Figure \ref{AMT2}(a) the dominoes containing matched-upward black squares. The QAR (rotated by $45^{\circ}$) can be partitioned into $4k-1$ rows of squares; and we call these rows black or white depending on the color of their squares.

We now describe the AMT $\tau:=\Phi(T)$. We label all black squares on each row by $1,2,\dots,n$ from left to right (we also label all black squares removed on the bottommost row). Denote by $B(i,j)$ the black square at the position $j$ on the $i$-th row (from the top). The positive entries in the $i$-th row of $\tau$ are the labels (positions) of the matched-upward squares on the $i$-th black row. By the antisymmetry, $\tau$ is completely determined by its positive entries (see the illustration in Figure \ref{AMT2}(b); the positive entries of the AMT are restricted inside the dotted right triangle).

\medskip

We would like to show that $\Phi$ is well defined, i.e. we will verify that $\tau=\Phi(T)\in \mathcal{A}_{2k}(a_1,a_2,\dotsc,a_k)$. Consider the strip consisting of $2(i-1)$ top rows in the QAR, for $i\leq 2k$. There are $(i-1)n$ black squares and $(i-1)n+\lfloor\frac{i}{2}\rfloor$ white squares in this strip. Thus, there are $\lfloor\frac{i}{2}\rfloor$ white squares in the strip, which are matched with some black squares outside the strip. These must be $\lfloor\frac{i}{2}\rfloor$ white squares on the $(i-1)$-th white row, and those white squares are matched with some black squares on the $i$-th black row. It means that there are $\lfloor\frac{i}{2}\rfloor$ matched-upward black squares on the $i$-th black row; equivalently, the $i$-th row of $\tau$ has exactly $\lfloor\frac{i}{2}\rfloor$ positive entries. 

Since each row of $\tau$ records the positions of matched-upward squares in a black row, the entries in each row of $\tau$ are strictly increasing from left to right. We now want to verify that the entries in each diagonal are weakly increasing from left to right. Since $\tau$ is antisymmetric, we only need to prove this monotonicity for the right half of $\tau$, which contains all positive entries. Assume that the $i$-th row of $\tau$ has the positive entries $0<t_{i,1}<t_{i,2}<\dots<t_{i,l(i)}$, where $l(i)=\lfloor\frac{i}{2}\rfloor$. We only need to show that
\begin{equation}\label{ineq1}
t_{i-1,1}\leq t_{i,1}\leq t_{i-1,2}\leq t_{i,2}\leq \dots\leq t_{i-1,l(i-1)}\leq t_{i,l(i)}
\end{equation}
for any odd $i$, and that
\begin{equation}\label{ineq2}
t_{i,1}\leq t_{i-1,1}\leq t_{i,2}\leq t_{i-1,2}\leq t_{i,3}\leq  \dots\leq t_{i-1,l(i-1)}\leq t_{i,l(i)}
\end{equation}
for any even $i$.

We consider first the case when $i$ is odd. Consider the first inequality $t_{i-1,1}\leq t_{i,1}$. We assume otherwise that $t_{i-1,1}> t_{i,1}$, then by definition, all (black) squares $B(i-1,1), B(i-1,2),\dotsc,B(i-1,t_{i,1})$ are matched downward. However, the square $B(i,t_{i,1})$ cannot be matched upward, since all white neighbor squares above it are already matched with other black squares, a contradiction (see Figure \ref{AMT3}(a) for $t_{i,1}=5$; the upper black row is the $(i-1)$-th one; the matched-downward squares on the $i$-th black row are not shown in the picture; and the square having black core indicates the square that cannot be matched).

\begin{figure}\centering
\setlength{\unitlength}{3947sp}%
\begingroup\makeatletter\ifx\SetFigFont\undefined%
\gdef\SetFigFont#1#2#3#4#5{%
  \reset@font\fontsize{#1}{#2pt}%
  \fontfamily{#3}\fontseries{#4}\fontshape{#5}%
  \selectfont}%
\fi\endgroup%
\resizebox{14cm}{!}{
\begin{picture}(0,0)%
\includegraphics{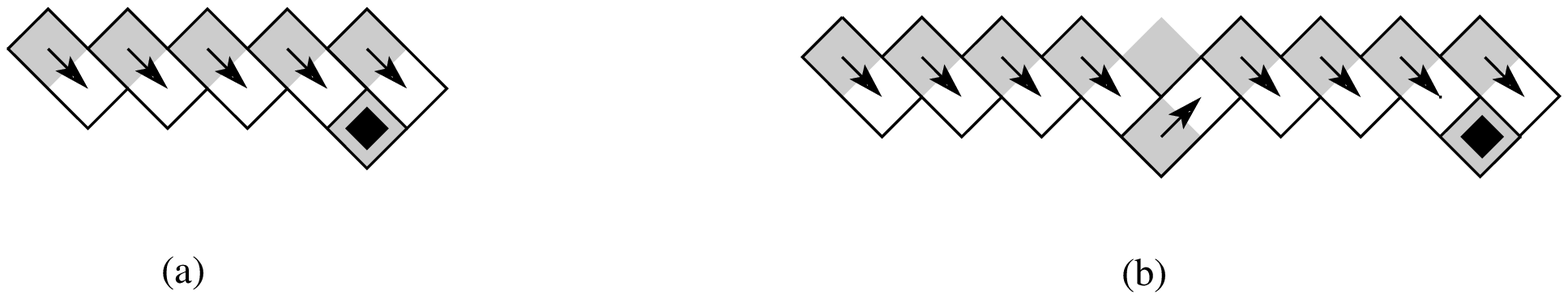}%
\end{picture}%
%
%

\begin{picture}(9999,2126)(631,-3714)
\put(3241,-3271){\makebox(0,0)[lb]{\smash{{\SetFigFont{14}{16.8}{\rmdefault}{\mddefault}{\itdefault}{\color[rgb]{0,0,0}$t_{i,1}=5$}%
}}}}
\put(1606,-1801){\makebox(0,0)[lb]{\smash{{\SetFigFont{14}{16.8}{\rmdefault}{\mddefault}{\itdefault}{\color[rgb]{0,0,0}$1$}%
}}}}
\put(2101,-1801){\makebox(0,0)[lb]{\smash{{\SetFigFont{14}{16.8}{\rmdefault}{\mddefault}{\itdefault}{\color[rgb]{0,0,0}$2$}%
}}}}
\put(2521,-1801){\makebox(0,0)[lb]{\smash{{\SetFigFont{14}{16.8}{\rmdefault}{\mddefault}{\itdefault}{\color[rgb]{0,0,0}$3$}%
}}}}
\put(3001,-1801){\makebox(0,0)[lb]{\smash{{\SetFigFont{14}{16.8}{\rmdefault}{\mddefault}{\itdefault}{\color[rgb]{0,0,0}$4$}%
}}}}
\put(3511,-1801){\makebox(0,0)[lb]{\smash{{\SetFigFont{14}{16.8}{\rmdefault}{\mddefault}{\itdefault}{\color[rgb]{0,0,0}$5$}%
}}}}
\put(781,-2776){\makebox(0,0)[lb]{\smash{{\SetFigFont{14}{16.8}{\rmdefault}{\mddefault}{\itdefault}{\color[rgb]{0,0,0}$i$}%
}}}}
\put(646,-2356){\makebox(0,0)[lb]{\smash{{\SetFigFont{14}{16.8}{\rmdefault}{\mddefault}{\itdefault}{\color[rgb]{0,0,0}$i-1$}%
}}}}
\put(9743,-3268){\makebox(0,0)[lb]{\smash{{\SetFigFont{14}{16.8}{\rmdefault}{\mddefault}{\itdefault}{\color[rgb]{0,0,0}$t_{i,2}=9$}%
}}}}
\put(7890,-3298){\makebox(0,0)[lb]{\smash{{\SetFigFont{14}{16.8}{\rmdefault}{\mddefault}{\itdefault}{\color[rgb]{0,0,0}$t_{i,1}=5$}%
}}}}
\put(6270,-1801){\makebox(0,0)[lb]{\smash{{\SetFigFont{14}{16.8}{\rmdefault}{\mddefault}{\itdefault}{\color[rgb]{0,0,0}$1$}%
}}}}
\put(6757,-1801){\makebox(0,0)[lb]{\smash{{\SetFigFont{14}{16.8}{\rmdefault}{\mddefault}{\itdefault}{\color[rgb]{0,0,0}$2$}%
}}}}
\put(7215,-1801){\makebox(0,0)[lb]{\smash{{\SetFigFont{14}{16.8}{\rmdefault}{\mddefault}{\itdefault}{\color[rgb]{0,0,0}$3$}%
}}}}
\put(8175,-1801){\makebox(0,0)[lb]{\smash{{\SetFigFont{14}{16.8}{\rmdefault}{\mddefault}{\itdefault}{\color[rgb]{0,0,0}$5$}%
}}}}
\put(8655,-1801){\makebox(0,0)[lb]{\smash{{\SetFigFont{14}{16.8}{\rmdefault}{\mddefault}{\itdefault}{\color[rgb]{0,0,0}$6$}%
}}}}
\put(9120,-1801){\makebox(0,0)[lb]{\smash{{\SetFigFont{14}{16.8}{\rmdefault}{\mddefault}{\itdefault}{\color[rgb]{0,0,0}$7$}%
}}}}
\put(9592,-1801){\makebox(0,0)[lb]{\smash{{\SetFigFont{14}{16.8}{\rmdefault}{\mddefault}{\itdefault}{\color[rgb]{0,0,0}$8$}%
}}}}
\put(10065,-1801){\makebox(0,0)[lb]{\smash{{\SetFigFont{14}{16.8}{\rmdefault}{\mddefault}{\itdefault}{\color[rgb]{0,0,0}$9$}%
}}}}
\put(5542,-2367){\makebox(0,0)[lb]{\smash{{\SetFigFont{14}{16.8}{\rmdefault}{\mddefault}{\itdefault}{\color[rgb]{0,0,0}$i-1$}%
}}}}
\put(5625,-2817){\makebox(0,0)[lb]{\smash{{\SetFigFont{14}{16.8}{\rmdefault}{\mddefault}{\itdefault}{\color[rgb]{0,0,0}$i$}%
}}}}
\put(7687,-1801){\makebox(0,0)[lb]{\smash{{\SetFigFont{14}{16.8}{\rmdefault}{\mddefault}{\itdefault}{\color[rgb]{0,0,0}$4$}%
}}}}
\end{picture}}
\caption{Illustrating the proof of the inequality (\ref{ineq1}).}
\label{AMT3}
\end{figure}

We now show that 
\begin{claim}
$t_{i,1}\leq t_{i-1,2}\leq t_{i,2}$ if $l(i)\geq 2$. 
\end{claim}

\begin{proof}There are two cases to distinguish.

\medskip

\begin{figure}\centering
\setlength{\unitlength}{3947sp}%
\begingroup\makeatletter\ifx\SetFigFont\undefined%
\gdef\SetFigFont#1#2#3#4#5{%
  \reset@font\fontsize{#1}{#2pt}%
  \fontfamily{#3}\fontseries{#4}\fontshape{#5}%
  \selectfont}%
\fi\endgroup%
\resizebox{14cm}{!}{
\begin{picture}(0,0)%
\includegraphics{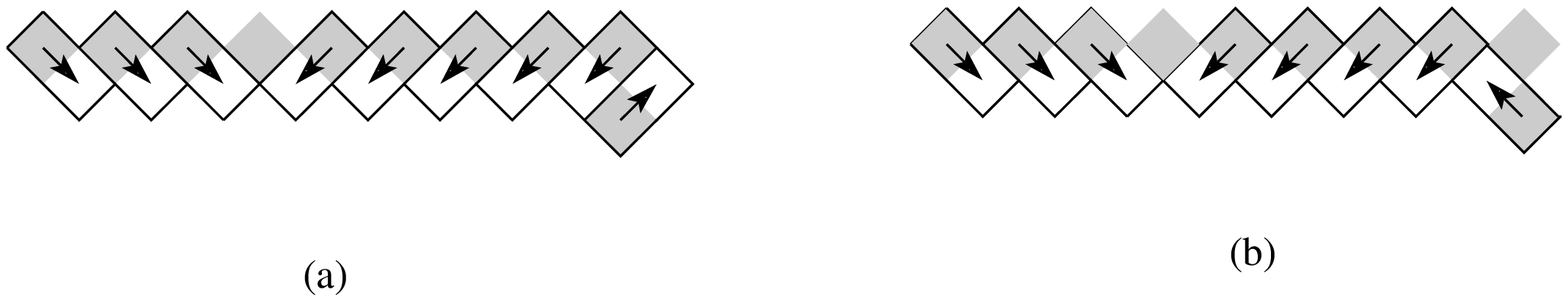}%
\end{picture}%
%
%

\begin{picture}(10992,2604)(133,-2032)
\put(7821,340){\makebox(0,0)[lb]{\smash{{\SetFigFont{14}{16.8}{\rmdefault}{\mddefault}{\updefault}{\color[rgb]{0,0,0}$t_{i-1,1}=4$}%
}}}}
\put(10356,-1475){\makebox(0,0)[lb]{\smash{{\SetFigFont{14}{16.8}{\rmdefault}{\mddefault}{\updefault}{\color[rgb]{0,0,0}$t_{i,1}=9$}%
}}}}
\put(6797,-26){\makebox(0,0)[lb]{\smash{{\SetFigFont{14}{16.8}{\rmdefault}{\mddefault}{\itdefault}{\color[rgb]{0,0,0}$1$}%
}}}}
\put(7284,-26){\makebox(0,0)[lb]{\smash{{\SetFigFont{14}{16.8}{\rmdefault}{\mddefault}{\itdefault}{\color[rgb]{0,0,0}$2$}%
}}}}
\put(7742,-34){\makebox(0,0)[lb]{\smash{{\SetFigFont{14}{16.8}{\rmdefault}{\mddefault}{\itdefault}{\color[rgb]{0,0,0}$3$}%
}}}}
\put(8214,-26){\makebox(0,0)[lb]{\smash{{\SetFigFont{14}{16.8}{\rmdefault}{\mddefault}{\itdefault}{\color[rgb]{0,0,0}$4$}%
}}}}
\put(8702,-26){\makebox(0,0)[lb]{\smash{{\SetFigFont{14}{16.8}{\rmdefault}{\mddefault}{\itdefault}{\color[rgb]{0,0,0}$5$}%
}}}}
\put(9182,-26){\makebox(0,0)[lb]{\smash{{\SetFigFont{14}{16.8}{\rmdefault}{\mddefault}{\itdefault}{\color[rgb]{0,0,0}$6$}%
}}}}
\put(9647,-26){\makebox(0,0)[lb]{\smash{{\SetFigFont{14}{16.8}{\rmdefault}{\mddefault}{\itdefault}{\color[rgb]{0,0,0}$7$}%
}}}}
\put(10119,-26){\makebox(0,0)[lb]{\smash{{\SetFigFont{14}{16.8}{\rmdefault}{\mddefault}{\itdefault}{\color[rgb]{0,0,0}$8$}%
}}}}
\put(10592,-26){\makebox(0,0)[lb]{\smash{{\SetFigFont{14}{16.8}{\rmdefault}{\mddefault}{\itdefault}{\color[rgb]{0,0,0}$9$}%
}}}}
\put(6062,-455){\makebox(0,0)[lb]{\smash{{\SetFigFont{14}{16.8}{\rmdefault}{\mddefault}{\itdefault}{\color[rgb]{0,0,0}$i-1$}%
}}}}
\put(6302,-912){\makebox(0,0)[lb]{\smash{{\SetFigFont{14}{16.8}{\rmdefault}{\mddefault}{\itdefault}{\color[rgb]{0,0,0}$i$}%
}}}}
\put(1907,317){\makebox(0,0)[lb]{\smash{{\SetFigFont{14}{16.8}{\rmdefault}{\mddefault}{\updefault}{\color[rgb]{0,0,0}$t_{i-1,1}=4$}%
}}}}
\put(4442,-1498){\makebox(0,0)[lb]{\smash{{\SetFigFont{14}{16.8}{\rmdefault}{\mddefault}{\updefault}{\color[rgb]{0,0,0}$t_{i,1}=9$}%
}}}}
\put(883,-26){\makebox(0,0)[lb]{\smash{{\SetFigFont{14}{16.8}{\rmdefault}{\mddefault}{\itdefault}{\color[rgb]{0,0,0}$1$}%
}}}}
\put(1370,-26){\makebox(0,0)[lb]{\smash{{\SetFigFont{14}{16.8}{\rmdefault}{\mddefault}{\itdefault}{\color[rgb]{0,0,0}$2$}%
}}}}
\put(1828,-57){\makebox(0,0)[lb]{\smash{{\SetFigFont{14}{16.8}{\rmdefault}{\mddefault}{\itdefault}{\color[rgb]{0,0,0}$3$}%
}}}}
\put(2300,-26){\makebox(0,0)[lb]{\smash{{\SetFigFont{14}{16.8}{\rmdefault}{\mddefault}{\itdefault}{\color[rgb]{0,0,0}$4$}%
}}}}
\put(2788,-26){\makebox(0,0)[lb]{\smash{{\SetFigFont{14}{16.8}{\rmdefault}{\mddefault}{\itdefault}{\color[rgb]{0,0,0}$5$}%
}}}}
\put(3268,-26){\makebox(0,0)[lb]{\smash{{\SetFigFont{14}{16.8}{\rmdefault}{\mddefault}{\itdefault}{\color[rgb]{0,0,0}$6$}%
}}}}
\put(3733,-26){\makebox(0,0)[lb]{\smash{{\SetFigFont{14}{16.8}{\rmdefault}{\mddefault}{\itdefault}{\color[rgb]{0,0,0}$7$}%
}}}}
\put(4205,-26){\makebox(0,0)[lb]{\smash{{\SetFigFont{14}{16.8}{\rmdefault}{\mddefault}{\itdefault}{\color[rgb]{0,0,0}$8$}%
}}}}
\put(4678,-26){\makebox(0,0)[lb]{\smash{{\SetFigFont{14}{16.8}{\rmdefault}{\mddefault}{\itdefault}{\color[rgb]{0,0,0}$9$}%
}}}}
\put(148,-478){\makebox(0,0)[lb]{\smash{{\SetFigFont{14}{16.8}{\rmdefault}{\mddefault}{\itdefault}{\color[rgb]{0,0,0}$i-1$}%
}}}}
\put(388,-935){\makebox(0,0)[lb]{\smash{{\SetFigFont{14}{16.8}{\rmdefault}{\mddefault}{\itdefault}{\color[rgb]{0,0,0}$i$}%
}}}}
\end{picture}}
\caption{Illustrating the proof of the inequality (\ref{ineq1}) (cont.).}
\label{AMT4}
\end{figure}

\textit{Case 1. $t_{i-1,1}=t_{i,1}$.}\\

Since $t_{i,1}=t_{i-1,1}<t_{i-1,2}$, we only need to show that $t_{i-1,2}\leq t_{i,2}$. Assume otherwise that $t_{i-1,2}> t_{i,2}$, then same situation as in the previous paragraph happens with the square $B(i,t_{i,2})$: all of its upper white neighbors are matched with other black squares, so it cannot be matched upward, a contradiction (see Figure \ref{AMT3}(b) for the case $t_{i,1}=5$ and $t_{i,2}=9$; the square having black core indicates the one cannot be matched).\\

\textit{Case 2.  $t_{i-1,1}<t_{i,1}$.}\\

One can see that the dominoes containing the black squares $B(i-1,t_{i-1,1}+1),B(i-1,t_{i-1,2}+2),\dotsc,B(i-1,t_{i,1}-1)$ are all forced, and  these black squares are all matched downward (see Figures \ref{AMT4}(a) and (b) for the two possibility in the case when $t_{i-1,1}=5$ and $t_{i,1}=9$). Thus, the second match-upward square in the $(i-1)$-th black row must be on the right of $B(i-1,t_{i,1}-1)$, this means that $t_{i-1,2}\geq t_{i,1}$. Thus,  we only need to show that $t_{i-1,2}\leq t_{i,2}$. Again, assume otherwise that we have the opposite inequality $t_{i-1,2}> t_{i,2}$. In this case, the square $B(i,t_{i,2})$ cannot be matched upward, a contradiction (see the two possible cases in Figure \ref{AMT5}).
\end{proof}

\begin{figure}\centering
\setlength{\unitlength}{3947sp}%
\begingroup\makeatletter\ifx\SetFigFont\undefined%
\gdef\SetFigFont#1#2#3#4#5{%
  \reset@font\fontsize{#1}{#2pt}%
  \fontfamily{#3}\fontseries{#4}\fontshape{#5}%
  \selectfont}%
\fi\endgroup%
\resizebox{11cm}{!}{
\begin{picture}(0,0)%
\includegraphics{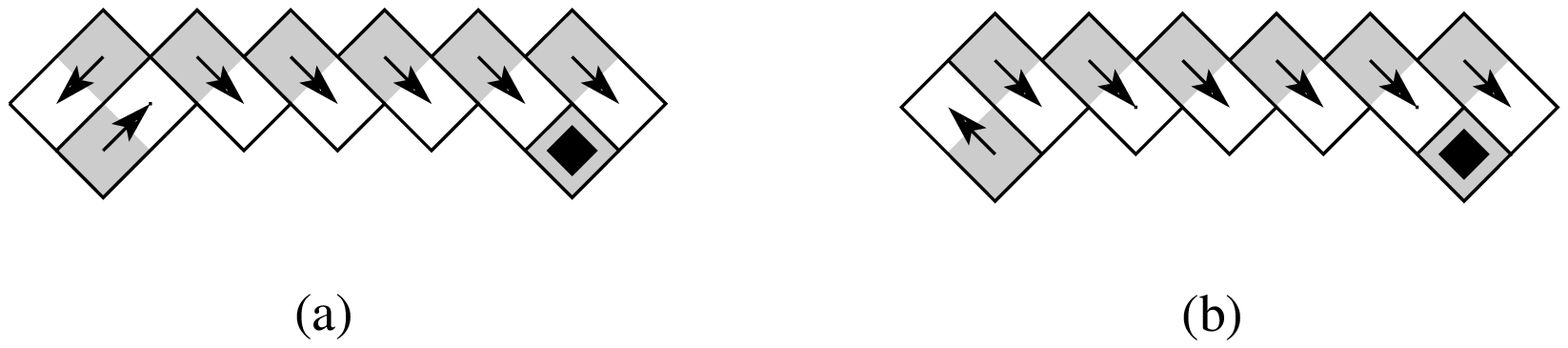}%
\end{picture}%
%

\begin{picture}(8423,2159)(346,-2762)
\put(4880,-1374){\makebox(0,0)[lb]{\smash{{\SetFigFont{14}{16.8}{\rmdefault}{\mddefault}{\itdefault}{\color[rgb]{0,0,0}$i-1$}%
}}}}
\put(1104,-2319){\makebox(0,0)[lb]{\smash{{\SetFigFont{14}{16.8}{\rmdefault}{\mddefault}{\itdefault}{\color[rgb]{0,0,0}$t_{i,1}=4$}%
}}}}
\put(3421,-2334){\makebox(0,0)[lb]{\smash{{\SetFigFont{14}{16.8}{\rmdefault}{\mddefault}{\itdefault}{\color[rgb]{0,0,0}$t_{i,2}=9$}%
}}}}
\put(1854,-834){\makebox(0,0)[lb]{\smash{{\SetFigFont{14}{16.8}{\rmdefault}{\mddefault}{\itdefault}{\color[rgb]{0,0,0}$5$}%
}}}}
\put(2334,-834){\makebox(0,0)[lb]{\smash{{\SetFigFont{14}{16.8}{\rmdefault}{\mddefault}{\itdefault}{\color[rgb]{0,0,0}$6$}%
}}}}
\put(2799,-834){\makebox(0,0)[lb]{\smash{{\SetFigFont{14}{16.8}{\rmdefault}{\mddefault}{\itdefault}{\color[rgb]{0,0,0}$7$}%
}}}}
\put(3271,-834){\makebox(0,0)[lb]{\smash{{\SetFigFont{14}{16.8}{\rmdefault}{\mddefault}{\itdefault}{\color[rgb]{0,0,0}$8$}%
}}}}
\put(3744,-834){\makebox(0,0)[lb]{\smash{{\SetFigFont{14}{16.8}{\rmdefault}{\mddefault}{\itdefault}{\color[rgb]{0,0,0}$9$}%
}}}}
\put(1854,-834){\makebox(0,0)[lb]{\smash{{\SetFigFont{14}{16.8}{\rmdefault}{\mddefault}{\itdefault}{\color[rgb]{0,0,0}$5$}%
}}}}
\put(2334,-834){\makebox(0,0)[lb]{\smash{{\SetFigFont{14}{16.8}{\rmdefault}{\mddefault}{\itdefault}{\color[rgb]{0,0,0}$6$}%
}}}}
\put(1359,-834){\makebox(0,0)[lb]{\smash{{\SetFigFont{14}{16.8}{\rmdefault}{\mddefault}{\itdefault}{\color[rgb]{0,0,0}$4$}%
}}}}
\put(361,-1374){\makebox(0,0)[lb]{\smash{{\SetFigFont{14}{16.8}{\rmdefault}{\mddefault}{\itdefault}{\color[rgb]{0,0,0}$i-1$}%
}}}}
\put(451,-1854){\makebox(0,0)[lb]{\smash{{\SetFigFont{14}{16.8}{\rmdefault}{\mddefault}{\itdefault}{\color[rgb]{0,0,0}$i$}%
}}}}
\put(6313,-834){\makebox(0,0)[lb]{\smash{{\SetFigFont{14}{16.8}{\rmdefault}{\mddefault}{\itdefault}{\color[rgb]{0,0,0}$5$}%
}}}}
\put(6793,-834){\makebox(0,0)[lb]{\smash{{\SetFigFont{14}{16.8}{\rmdefault}{\mddefault}{\itdefault}{\color[rgb]{0,0,0}$6$}%
}}}}
\put(7258,-834){\makebox(0,0)[lb]{\smash{{\SetFigFont{14}{16.8}{\rmdefault}{\mddefault}{\itdefault}{\color[rgb]{0,0,0}$7$}%
}}}}
\put(7730,-834){\makebox(0,0)[lb]{\smash{{\SetFigFont{14}{16.8}{\rmdefault}{\mddefault}{\itdefault}{\color[rgb]{0,0,0}$8$}%
}}}}
\put(8203,-834){\makebox(0,0)[lb]{\smash{{\SetFigFont{14}{16.8}{\rmdefault}{\mddefault}{\itdefault}{\color[rgb]{0,0,0}$9$}%
}}}}
\put(6313,-834){\makebox(0,0)[lb]{\smash{{\SetFigFont{14}{16.8}{\rmdefault}{\mddefault}{\itdefault}{\color[rgb]{0,0,0}$5$}%
}}}}
\put(6793,-834){\makebox(0,0)[lb]{\smash{{\SetFigFont{14}{16.8}{\rmdefault}{\mddefault}{\itdefault}{\color[rgb]{0,0,0}$6$}%
}}}}
\put(5818,-834){\makebox(0,0)[lb]{\smash{{\SetFigFont{14}{16.8}{\rmdefault}{\mddefault}{\itdefault}{\color[rgb]{0,0,0}$4$}%
}}}}
\put(5570,-2319){\makebox(0,0)[lb]{\smash{{\SetFigFont{14}{16.8}{\rmdefault}{\mddefault}{\itdefault}{\color[rgb]{0,0,0}$t_{i,1}=4$}%
}}}}
\put(7903,-2296){\makebox(0,0)[lb]{\smash{{\SetFigFont{14}{16.8}{\rmdefault}{\mddefault}{\itdefault}{\color[rgb]{0,0,0}$t_{i,2}=9$}%
}}}}
\put(5075,-1861){\makebox(0,0)[lb]{\smash{{\SetFigFont{14}{16.8}{\rmdefault}{\mddefault}{\itdefault}{\color[rgb]{0,0,0}$i$}%
}}}}
\end{picture}}
\caption{Illustrating the proof of the inequality (\ref{ineq1}) (cont.).}
\label{AMT5}
\end{figure}

By considering similarly the situations when $t_{i-1,2}=t_{i,2}$ or $t_{i-1,2}<t_{i,2}$, we obtain $t_{i,2}\leq t_{i-1,3}\leq t_{i,3}$ if $l(i)\geq 3$. Keep doing this process, we get the remaining inequalities in (\ref{ineq1}).

\medskip

We now consider the case when $i$ is even. The leftmost inequality in (\ref{ineq2}), $t_{i,1}\leq t_{i-1,1}$, follows easily from considering forced dominoes. Similar to the case of odd $i$, by considering two cases when $t_{i,1}=t_{i-1,1}$ and when $t_{i,1}<t_{i-1,1}$, we get now the inequality $t_{i-1,1}\leq t_{i,2}\leq t_{i-1,2}$. Then (\ref{ineq2}) is obtained by applying this argument repeatedly.  

In summary, $\tau=\Phi(T)$ is indeed an AMT of order $2k$. Moreover, since all black squares on the bottom of the QAR must be matched upward, the positive entries in the bottommost row of $\tau$ are $a_1,a_2,\dots,a_k$. Therefore, $\tau\in  \mathcal{A}_{2k}(a_1,a_2,\dotsc,a_k)$, and the map $\Phi$ is well defined. 

\medskip

\begin{figure}\centering
\includegraphics[width=14cm]{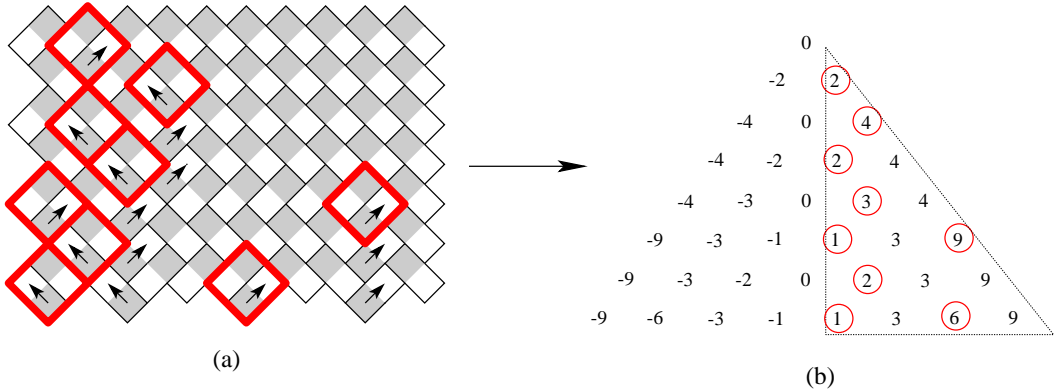}%
\caption{The $2\times 2$-blocks in the domino tilings of $RE_{7,10}(2,3,6,9)$ corresponding to the elements of $S(\tau)$ (the circled entries on the right).}
\label{AMT2b}
\end{figure}

Next, we want to compute the cardinality of the set $\Phi^{-1}(\tau)$ for any given AMT $\tau \in \mathcal{A}_{2k}(a_1,a_2,\dotsc,a_k)$, i.e. we  want to find out how many different domino tilings of $RE_{2k-1,n}(a_1,a_2,\dotsc,a_k)$ corresponding to  $\tau$. Let $S(\tau)$ be the set of positive entries of $\tau$ which do not appear in the row above (see the circled entries in Figure  \ref{AMT2b}(b)). We assign to each black square of $RE_{2k-1,n}(a_1,a_2,\dotsc,a_k)$ a label $\U$ or $\D$, so that  only the black squares $B(i,t_{i,j})$ have label $\U$, and all other ones have label $\D$. A tiling $T$ corresponds to $\tau$ if and only if all $\U$-squares are matched upward and all $\D$-squares are matched downward in $T$.

View the QAR $RE_{2k-1,n}(a_1,a_2,\dotsc,a_k)$ as the union of $2n+1$ columns of black or white squares. Imagine that we are dropping dominoes to cover successively the squares in the first black column, so that the ones with label $\U$ (resp., $\D$) are matched upward (resp., downward). One can see that all the dominoes are forced, except for those cases when a $\U$-square stays right below a $\D$-square (this correspond to a $1$ in one row of $\tau$ but not in the row above it).  At each exceptional place, we have two ways cover these two black squares by two parallel dominoes, which create a $2\times2$ block (see the $2\times 2$-blocks restricted by the bold contours in Figure \ref{AMT2b}(a)).  If we now drop dominoes to cover the squares in the second black column, all dominoes are forced, except for some $2\times2$-block corresponding to a $2$ that appears in some row of $\tau$ but not in the preceding row. Continuing in this way, one can see that the whole tiling is forced, except for certain $2\times2$-blocks corresponding to the elements of $S(\tau)$. Since each block can be covered in two ways, the number of tilings $T$ corresponding to $\tau$ is $2^{|S(\tau)|}$.  In particular, $\Phi$ is surjective; and we obtain
 \begin{align}
|\mathcal{T}(RE_{2k-1,n}(a_1,a_2,\dotsc,a_k))|&=\sum_{\tau\in\mathcal{A}_{2k}(a_1,\dots,a_k)}|\Phi^{-1}(\tau)|\\
 &=\sum_{\tau\in\mathcal{A}_{2k}(a_1,\dots,a_k)}2^{|S(\tau)|}\\
 &=A^{2}_{2k}(a_1,\dots,a_k),
 \end{align}
which implies (\ref{main1}).

\bigskip

Next, we prove the equality (\ref{main3}). Similar to (\ref{main1}),  the first equality follows from the fact that $TE_{2k,n}(a_1,a_2,\dotsc,a_{k})$ is obtained by removing forced dominoes on the top of $TE_{2k+1,n}(a_1,a_2,\dotsc,a_{k})$. Thus, we only need to show that
\[\T(TE_{2k,n}(a_1,a_2,\dotsc,a_{k}))=2^{-k}A^{2}_{2k+1}(a_1,a_2,\dotsc,a_k).\]

\begin{figure}\centering
\includegraphics[width=14cm]{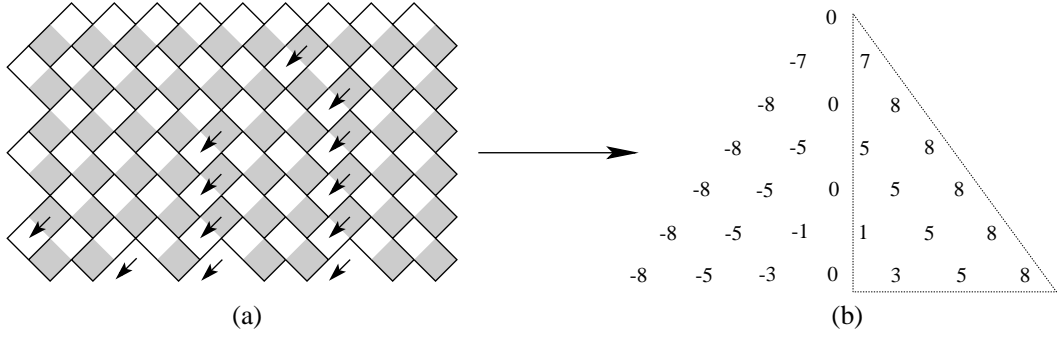}%
\caption{A domino tilings of $TE_{6,10}(3,5,8)$ and the corresponding AMT.}
\label{AMT6}
\end{figure}

We construct a map
\[\Psi:\mathcal{T}(TE_{2k,n}(a_1,a_2,\dotsc,a_k))\rightarrow\mathcal{A}_{2k+1}(a_1,a_2,\dotsc,a_k)\]
 similar to $\Phi$ in the proof of (\ref{main1}), the only difference is that  the $i$-th row of the AMT $\Psi(T)$ records the positions of the squares on the $i$-th black row, which are \textit{not} matched upward (\textit{including the squares removed on the bottom of  $TE_{2k,n}(a_1,a_2,\dotsc,a_k)$}). Figure \ref{AMT6} illustrates the map $\Psi$ for $k=3$, $n=10$, $a_1=3$, $a_2=5$, and $a_3=8$. Similar to the equality (\ref{main1}), one can verify that $\Psi(T)$ is indeed an AMT in $\mathcal{A}_{2k+1}(a_1,a_2,\dotsc,a_k)$, i.e. $\Psi$ is well defined.

Similarly, given an AMT $\tau\in\mathcal{A}_{2k+1}(a_1,a_2,\dotsc,a_k)$. By doing the same domino-dropping process, we get that there are now $2^{|V(\tau)|}$ tilings of $TE_{2k,n}(a_1,a_2,\dotsc,a_{k})$ corresponding to $\tau$, where $V(\tau)$ is the set of positive entries of $\tau$, which do not appear in the row \textit{below}. It means that $|\Psi^{-1}(\tau)|=2^{|V(\tau)|}$, and $\Psi$ is surjective.

We now compare the cardinalities of two sets $S:=S(\tau)$ and $V:=V(\tau)$. Denote by $r_i$ the $i$-th row of $\tau$, and $V_i:=V(\tau)\cap r_i$ and $S_i:=S(\tau)\cap r_i$.  Partition $\tau$ into two sets $\tau_{E}$, the set of all entries on the even rows, and $\tau_{O}$, the set of all entries on the odd rows. We will compare the numbers of elements of $V$ and $S$ in each of these subsets. 

Remove the positive entries, which appear in both rows $r_{2i-1}$ and $r_{2i}$. By the definition, the remaining positive entries  on $r_{2i-1}$ are in $V_{2i-1}$ and the remaining positive ones on $r_{2i}$ are in $S_{2i}$. Since $r_{2i}$ has one positive entry more than  $r_{2i-1}$ (since each row $r_j$ has $l(j)$ positive entries, and $l(2i)=i=l(2i-1)+1$), we have $|S_{2i}|=|V_{2i-1}|+1$, for $i=1,2,\dots,k$. Adding up all latter equalities for $i=1,2,\dots,k$, we have $|S\cap\tau_E|=|V\cap\tau_O|+k$. Similarly, we have $|V_{2i}|=|S_{2i+1}|$, for $i=1,2,\dotsc,k$, so $|S\cap \tau_{O}|=|V\cap \tau_{E}|$. This implies that 
\begin{equation}\label{relate}
|V|=|V\cap\tau_{E}|+|V\cap \tau_{O}|=|S\cap \tau_{O}|+|S\cap \tau_{E}|-k=|S|-k.
\end{equation}

By (\ref{relate}), we have
 \begin{align}
|\mathcal{T}(TE_{2k,n}(a_1,a_2,\dotsc,a_k))|&=\sum_{\tau\in\mathcal{A}_{2k+1}(a_1,\dots,a_k)}|\Psi^{-1}(\tau)|\\
 &=\sum_{\tau\in\mathcal{A}_{2k+1}(a_1,\dots,a_k)}2^{|S(\tau)|-k}\\
 &=2^{-k}A^{2}_{2k+1}(a_1,\dots,a_k).
 \end{align}
Then  (\ref{main3}) follows, and this finishes our proof.

\end{document}